\documentclass{article}
\usepackage{amsmath,amssymb,amsthm}
\usepackage{geometry}
\usepackage{graphicx}
\usepackage{authblk}
\usepackage[numbers,sort&compress]{natbib}
\usepackage[colorlinks=true,linkcolor=blue,citecolor=blue,urlcolor=blue]{hyperref}

\newtheorem{theorem}{Theorem}
\theoremstyle{definition}

\theoremstyle{remark}

\newcommand{\U}{\mathbf U}
\newcommand{\V}{\bar{\mathbf U}}
\newcommand{\A}{\mathbf A}
\newcommand{\B}{\bar{\mathbf A}}
\newcommand{\rank}{\operatorname{rank}}
\newcommand{\Sp}{\mathcal S_p}

\title{\large \bf A note on the rectangular Schoenberg--Whitney\\
theorem for clamped B-splines}
\author[a,*]{Shutao Tang}
\author[b]{Zixuan Zhao}
\affil[a]{\small Department of Mechanical Engineering, Northwestern Polytechnical University, Xi'an, China}
\affil[b]{\small The University of Chicago, 5801 S Ellis Ave, Chicago, IL 60637, United States}
\date{}

\begin{document}
\maketitle

\begin{abstract}
Wang et al.~(\emph{Computer-Aided Design} 40 (2008) 999--1008) proposed a knot-subsequence characterization of full row rank for
rectangular B-spline collocation matrices. We identify an incorrect inequality in their
Appendix proof, which invalidates the argument as written, and make explicit the endpoint
and multiplicity conventions on which the criterion depends. The defect lies in the proof,
not in the underlying result, which
is valid under suitably stated hypotheses; in the lofting construction of Wang et al. these
hypotheses are implicitly satisfied, so the lofted curves are unaffected. We state the
corrected result and give a new
proof, independent of the disputed argument, for clamped B-splines of arbitrary degree
with continuous interior multiplicities. 
In our proof, sufficiency follows from inclusion of spline spaces and knot insertion; 
necessity selects independent columns and builds the reduced knot
vector from them. The proof uses no recursive column-selection construction.

\vspace{0.5em}
\noindent\textbf{Keywords:} B-spline interpolation; Schoenberg--Whitney condition;
rectangular collocation matrix; knot subsequence; clamped knots
\end{abstract}
\section{Introduction}

A common task in geometric modeling is constructing a B-spline
curve (or surface) that passes through prescribed data while the degree and the knot
vector have already been fixed. Expressed in the B-spline basis, the task becomes a linear
system, and the practical question is whether enough independent spline functions are
available to accommodate every possible set of values. In other words, whether the system is
solvable for every data set is a question about the rank of the collocation matrix.

When the number of basis functions equals the number of data sites, the classical
Schoenberg--Whitney condition~\cite{deboor2001} characterizes the nonsingularity of
the square collocation matrix. Wang et al.~\cite{wang2008reducing} formulated an extension to the
rectangular case: a collocation matrix has full row rank precisely when a shorter knot
vector, selected from the occurrences of the original knot vector, satisfies the square
condition. This extension is the key to energy-minimizing B-spline lofting~\cite{park2003lofted}, where one common knot
vector must interpolate a whole family of section curves.

However, the proof of this claim in~\cite{wang2008reducing} is not valid as written,
and some of its conditions must be made explicit for the statement to be rigorous. The
underlying result is valid under suitably stated hypotheses, and Section~\ref{sec:proof}
establishes it by a different argument.
The present paper makes three contributions.\footnote{This work grew out of an attempt to
formalize the proof of~\cite{wang2008reducing} in the Lean proof assistant; the
formalization uncovered the defect reported below and led to the present paper.}
\begin{itemize}
\item \textbf{Error localization.} The proof in the Appendix of~\cite{wang2008reducing} is not
valid as written: it deduces a false inequality, and its treatment of clamped endpoints
and of repeated knots also requires qualification.
\item \textbf{A correct statement.} We state the corrected theorem with all hypotheses made
explicit: clamped knots and interior-knot multiplicities bounded by the degree.
\item \textbf{A new proof.} We give a proof independent of the disputed argument that does
not use its recursive column-selection construction. Sufficiency follows from
inclusion of spline spaces and knot insertion. Necessity selects independent columns and
then constructs the reduced knot vector.
\end{itemize}

The remainder of this paper is organized as follows. Section~\ref{sec:notation} states the
interpolation problem and introduces the clamped B-spline setting and notation.
Section~\ref{sec:revisited} revisits the published argument, identifying the incorrect
inequality and the convention issues with explicit counterexamples.
Section~\ref{sec:corrected} states the corrected theorem. Section~\ref{sec:proof} gives the
new proof. Section~\ref{sec:conclusion} concludes the paper.

\section{Problem definition and preliminaries}
\label{sec:notation}

\subsection{Setting and notation}

We adopt the notation of Wang et al.~\cite{wang2008reducing}. 
Fix integers $1\le p\le s\le m$ (with $p$ the spline degree) and distinct \emph{data parameters}
\[
\mathbf T=\{t_0,\dots,t_s\},\qquad 0=t_0<t_1<\cdots<t_s=1,
\]
and let $\mathbf Q_0,\dots,\mathbf Q_s$ be the corresponding data points. Choose a
clamped knot vector
\[
\U=(u_0\le\cdots\le u_{m+p+1}),\qquad
u_0=\cdots=u_p=0<u_{p+1},\qquad u_{m}<u_{m+1}=\cdots=u_{m+p+1}=1,
\]
whose interior knots lie in $(0,1)$ and have multiplicity at most $p$.
Let $N_{i,p}^{\U}$ be the degree-$p$ B-spline basis functions defined on the clamped knot vector $\U$, and let
\[
\Sp(\U)=\operatorname{span}\{N_{0,p}^{\U},\dots,N_{m,p}^{\U}\}
\]
be the associated spline space. The \emph{collocation matrix} (also called the
interpolation matrix) is
\[
\A=\bigl[N_{i,p}^{\U}(t_j)\bigr]_{\substack{0\le j\le s\\0\le i\le m}}
\in\mathbb R^{(s+1)\times(m+1)}.
\]
With control points $\mathbf P=[\mathbf P_0,\dots,\mathbf P_m]^{\!\top}$ and data
$\mathbf Q=[\mathbf Q_0,\dots,\mathbf Q_s]^{\!\top}$, interpolation means
\[
\A\mathbf P=\mathbf Q.
\]
When $s=m$, there is exactly one interpolating B-spline curve if and only if $\mathbf T$ and
$\U$ satisfy the Schoenberg--Whitney condition~\cite{deboor2001}. When $s<m$ the system is under-determined and a strain-energy minimizer~\cite{park2003lofted}
selects one interpolant; such constructions require the interpolation matrix $\A$ to have full row rank.

Unless explicitly stated otherwise in the examples of
Section~\ref{sec:revisited}, we use the continuous-spline convention: all interior multiplicities are at
most $p$, and at $t=1$ the basis is evaluated by its left limit. In particular,
\begin{equation}\label{eq:endpoint}
N_{i,p}^{\U}(0)=\delta_{i0},\qquad N_{i,p}^{\U}(1)=\delta_{im}.
\end{equation}
Here $\delta_{ij}$ denotes the Kronecker delta ($\delta_{ij}=1$ if $i=j$, and $0$
otherwise).

\subsection{The published statement}

In the setting introduced above, the claim under study is the following.

\begin{quote}
\hrule
\emph{Published statement (Wang et al.~\cite{wang2008reducing}).} The interpolation matrix
$\A$ is full-rank if and only if there exists a knot vector
$\V=\{\bar u_0,\dots,\bar u_{s+p+1}\}$ with $\bar u_i\in\U$ such that $\mathbf T$ and $\V$
satisfy the Schoenberg--Whitney condition.
\vspace{1.3mm}
\hrule
\end{quote}

By a \emph{reduced} knot vector we mean a clamped vector
\[
\V=(\bar u_0\le\cdots\le\bar u_{s+p+1}),\qquad
\bar u_0=\cdots=\bar u_p=0,\qquad
\bar u_{s+1}=\cdots=\bar u_{s+p+1}=1,
\]
obtained from $\U$ by deleting $m-s$ interior knot occurrences. Here
\emph{occurrence subsequence} means that order is preserved and that no knot value is
used more often than it occurs in $\U$; mere membership $\bar u_k\in\U$ does not express
this restriction. Deleting knot occurrences produces a spline subspace:
\[
\Sp(\V)\subseteq\Sp(\U).
\]
For a reduced vector $\V$ we write $\B=[N_{i,p}^{\V}(t_j)]_{j,i=0}^{s}$ for its square
collocation matrix.

With these notions in place, two conventions of the published statement must be made
explicit. First, the membership $\bar u_i\in\U$ is too
weak: it must mean that $\V$ is an \emph{occurrence subsequence} of $\U$, so that no knot
value is used more often than it occurs in $\U$. Second, the Schoenberg--Whitney condition
is read with the endpoint convention \eqref{eq:endpoint}; strict inequalities are imposed
only for the interior indices $1\le i\le s-1$, since the endpoint inequalities are
impossible for clamped knots.

\section{The published argument revisited}
\label{sec:revisited}

We now assess the published statement of Wang et al.~\cite{wang2008reducing} and its
Appendix proof. That proof contains an incorrect inequality, and its treatment of
endpoints and repeated knots also requires qualification.

\subsection{The conditions and their conventions}
\label{revisit:conditions}

In~\cite{wang2008reducing}, $P$ denotes the assertion that the interpolation matrix $\A$ has full row rank. 
Their Appendix rewrites the selected-knot condition as
\begin{equation}
\label{revisit:Q}
 \begin{aligned}
 Q:\quad &\exists\ f:\{0,\dots,s+p+1\}\longrightarrow
 \{0,\dots,m+p+1\}\ \text{strictly increasing},\\
 &u_{f(i)}<t_i<u_{f(i+p+1)}\quad(0\le i\le s),
 \end{aligned}
\end{equation}
and introduces
\begin{equation}
\label{revisit:R}
 \begin{aligned}
 R:\quad &\exists\ g:\{0,\dots,s\}\longrightarrow
 \{0,\dots,m\}\ \text{strictly increasing},\\
 &u_{g(i)}<t_i<u_{g(i)+p+1}\quad(0\le i\le s).
 \end{aligned}
\end{equation}
Selection here concerns knot \emph{occurrences} (Section~\ref{sec:notation}): the
multiplicity of a selected value cannot exceed its original multiplicity.

Conditions \eqref{revisit:Q}--\eqref{revisit:R} impose strict inequalities at every site.
In the standing endpoint-inclusive setting they are impossible: in particular,
$u_{f(0)}<t_0=0$ cannot hold. This conflicts with the interior-index treatment of the
square condition already used in Section~2 of~\cite{wang2008reducing}, where the two
endpoint inequalities are replaced by the clamped-endpoint convention.

For clarity, denote the endpoint-adjusted conditions by
\begin{equation}
\label{revisit:Q-cl}
 \begin{aligned}
 Q_{\mathrm{cl}}:\quad&
 \exists\ f:\{0,\dots,s+p+1\}\longrightarrow
 \{0,\dots,m+p+1\}\ \text{strictly increasing},\\
 &\V=(u_{f(0)},\dots,u_{f(s+p+1)})\ \text{clamped},\\
 &u_{f(i)}<t_i<u_{f(i+p+1)}\quad(1\le i\le s-1),
 \end{aligned}
\end{equation}
and
\begin{equation}
\label{revisit:R-cl}
 \begin{aligned}
 R_{\mathrm{cl}}:\quad&
 \exists\ g:\{0,\dots,s\}\longrightarrow
 \{0,\dots,m\}\ \text{strictly increasing},\\
 &g(0)=0,\quad g(s)=m,\\
 & u_{g(i)}<t_i<u_{g(i)+p+1}\quad(1\le i\le s-1).
 \end{aligned}
\end{equation}
Here clamping of $\V$ means retaining $p+1$ occurrences of each endpoint. These are
our explicit conventions, not the Appendix's literal all-indices formulations.

\subsection{The disputed construction}
\label{revisit:construction}

Given $f$, the Appendix introduces the local shifts and the recurrence in
equation~(11):
\begin{equation}
\label{revisit:recurrence}
\begin{aligned}
      j_i  & =\min\{h\ge0:u_{f(i)+h}<t_i<u_{f(i)+h+p+1}\}, \\
      a_i  & =f(i)+j_i,                                    \\
      g(0) & =a_0, \qquad g(i) =\max\{a_i,g(i-1)+1\} \quad(i>0).
\end{aligned}
\end{equation}
The minimum must be over nonnegative integers with valid knot indices, and its defining
set must first be shown nonempty.

In part~(1.3)(c), $c$ is the distance backwards from $k$ to the first index satisfying
$g(k-c)=a_{k-c}$. The final inequality chain includes
\begin{equation}
\label{revisit:false-chain}
 u_{f(k-c)+j_{k-c}+p+1}\le u_{f(k-c)}.
\end{equation}
This is incompatible with the defining support inequalities. By the induction hypothesis
of part~(1.3)(b), $u_{g(i)}<t_i<u_{g(i)+p+1}$ for every $i<k \; (k>1)$; since $k-c<k$,
substituting $i$ with $k-c$, we have
\[
 u_{g(k-c)}<t_{k-c}<u_{g(k-c)+p+1}.
\]
Using $g(k-c)=a_{k-c}$ and $a_{k-c}=f(k-c)+j_{k-c}$, this becomes
\[
 u_{f(k-c)}\le u_{a_{k-c}}<t_{k-c}<u_{a_{k-c}+p+1},
 \qquad\text{whence}\qquad
 u_{f(k-c)+j_{k-c}+p+1}>u_{f(k-c)},
\]
the reverse of \eqref{revisit:false-chain}, a contradiction.

To accommodate endpoint sites, we replace the minimum in
\eqref{revisit:recurrence} by the following endpoint-aware version:
\begin{equation}
\label{revisit:recurrence-cl}
 \begin{aligned}
 H_i&=\{h\ge0:  N_{f(i)+h,p}^{\U}(t_i)>0\},\\
 j_i&=\min H_i,\qquad a_i=f(i)+j_i,\\
 g(0)&=a_0,\qquad g(i)=\max\{a_i,g(i-1)+1\}\quad(i>0).
 \end{aligned}
\end{equation}
In the continuous class, the positivity test agrees with the original strict support
test at interior sites. At $0$ and $1$ it uses \eqref{eq:endpoint}. This is an
explicit modification of the construction of Wang et al.~\cite{wang2008reducing}.

We now give a concrete, normalized, clamped input that reaches the relevant branch:
\begin{equation}
      \label{revisit:reachable-example}
      \begin{aligned}
            p         & =1,\quad s=4,\quad m=5,                               \\
            \U        & =(0,0,\tfrac14,\tfrac12,\tfrac34,\tfrac{15}{16},1,1), \\
            \mathbf T & =(0,\tfrac18,\tfrac{13}{16},\tfrac78,1),\qquad
            f=(0,1,2,3,5,6,7).
      \end{aligned}
\end{equation}
The selected vector is
\[
 \V=(u_{f(0)},\dots,u_{f(6)}) = (0,0,\tfrac14,\tfrac12,\tfrac{15}{16},1,1),
\]
which is a clamped occurrence subsequence. It satisfies $Q_{\mathrm{cl}}$, since
\begin{align*}
 0=u_{f(1)} &< t_1=\tfrac18 < u_{f(3)}=\tfrac12,\\
 \tfrac14=u_{f(2)} &< t_2=\tfrac{13}{16} < u_{f(4)}=\tfrac{15}{16},\\
 \tfrac12=u_{f(3)} &< t_3=\tfrac78 < u_{f(5)}=1.
\end{align*}

\begin{table}[htbp]
      \centering
      \caption{The endpoint-adjusted construction \eqref{revisit:recurrence-cl} for the
            counterexample \eqref{revisit:reachable-example}.}
      \label{tab:reachable}
      \begin{tabular}{crrrrr}
            \hline
            $i$ & $f(i)$ & $t_i$   & $j_i$ & $a_i$ & $g(i)$ \\
            \hline
            0   & 0      & $0$     & 0     & 0     & 0      \\
            1   & 1      & $1/8$   & 0     & 1     & 1      \\
            2   & 2      & $13/16$ & 1     & 3     & 3      \\
            3   & 3      & $7/8$   & 0     & 3     & 4      \\
            4   & 5      & $1$     & 0     & 5     & 5      \\
            \hline
      \end{tabular}
\end{table}

All the shift sets are nonempty. At the endpoints $H_0=H_4=\{0\}$; in the interior,
$H_1=\{0\}$, $H_2=\{1,2\}$, and $H_3=\{0,1\}$. Thus
Table~\ref{tab:reachable} records the actual minimal shifts and recurrence values.

At $k=3$, $g(3)=4>a_3=3$. The first backward anchor is $g(2)=a_2=3$, so $c=1$, and we have $f(k-c)=f(2)=2$ and
$j_{k-c}=j_2=1$, so the left index of \eqref{revisit:false-chain} is
\[
f(k-c)+j_{k-c}+p+1=2+1+1+1=5,
\]
and the inequality reads $u_5\le u_2$, contradicting the fact that the
knot vector $\U$ is nondecreasing. The resulting
$g$ nevertheless satisfies $R_{\mathrm{cl}}$. Thus the example refutes the displayed inequality.

The collocation matrices are
\[
 \A=
 \begin{pmatrix}
       1   & 0   & 0 & 0   & 0   & 0 \\
       1/2 & 1/2 & 0 & 0   & 0   & 0 \\
       0   & 0   & 0 & 2/3 & 1/3 & 0 \\
       0   & 0   & 0 & 1/3 & 2/3 & 0 \\
       0   & 0   & 0 & 0   & 0   & 1
 \end{pmatrix},
 \qquad
 \B=
 \begin{pmatrix}
       1   & 0   & 0   & 0   & 0 \\
       1/2 & 1/2 & 0   & 0   & 0 \\
       0   & 0   & 2/7 & 5/7 & 0 \\
       0   & 0   & 1/7 & 6/7 & 0 \\
       0   & 0   & 0   & 0   & 1
 \end{pmatrix}.
\]
The minor of $\A$ in columns $0,1,3,4,5$ has determinant $1/6$, and
$\det\B=1/14$. Thus $\A$ has full row rank and $\B$ is nonsingular.
The endpoint fix therefore does not repair the displayed local inequality.

A further local correction is straightforward: the sentence preceding part~(1.1)
compares $t_i$ with the index $f(i+p+1)$ rather than the knot value $u_{f(i+p+1)}$.
Fixing this typo, however, does not repair the gap at \eqref{revisit:false-chain}.

\subsection{Endpoints and full-multiplicity knots}
\label{revisit:boundaries}

An endpoint example within the standing setting makes the obstruction concrete:
for $p=1$, $\U=(0,0,1,1)$, and $\mathbf T=(0,1)$, the collocation matrix is the
identity while the all-indices conditions $Q$ and $R$ are impossible.

We next temporarily admit interior knots of multiplicity $p+1$. These stress tests
leave the continuous-spline class. We use right-continuous values at interior breaks
and the left-limit extension at the final endpoint. For the first test we also omit
endpoint data sites:
\begin{equation}
\label{revisit:empty-minimum}
 p=1,\quad \U=(0,0,\tfrac12,\tfrac12,1,1),\quad
 \mathbf T=(\tfrac14,\tfrac12,\tfrac34),\qquad
 \V=(0,0,\tfrac12,1,1).
\end{equation}
$Q$ holds with $f=(0,1,2,4,5)$, and $\V$ is a clamped occurrence subsequence.
At the middle site, however, every full-vector open support is either $(0,1/2)$
or $(1/2,1)$. None contains $1/2$, so $R$ is false and the minimum in
\eqref{revisit:recurrence} does not exist. The collocation matrix is
\[
 \begin{pmatrix}
       1/2 & 1/2 & 0   & 0   \\
       0   & 0   & 1   & 0   \\
       0   & 0   & 1/2 & 1/2
 \end{pmatrix}.
\]
Its minor in columns $0,2,3$ has determinant $1/4$. Thus $P$ and $Q$
are true, but $R$ is false. This disproves the intermediate equivalence
$Q\Longleftrightarrow R$ in this discontinuous setting; it does not disprove
$P\Longleftrightarrow Q$ on this example.

Restoring endpoint sites does not justify retaining strict interior support
inequalities at discontinuities. Consider
\begin{equation}
\label{revisit:two-breaks}
 p=1,\quad \U=(0,0,\tfrac13,\tfrac13,\tfrac23,\tfrac23,1,1),\qquad
 \mathbf T=(0,\tfrac13,\tfrac12,\tfrac23,1).
\end{equation}
With the specified right-continuous interior values, the collocation matrix has rows
$e_0,e_2,(e_2+e_3)/2,e_4,e_5$ in $\mathbb R^6$. For arbitrary data
$(y_0,\dots,y_4)$, the control vector
$(y_0,0,y_1,2y_2-y_1,y_3,y_4)$ gives those data, proving full row rank. The minor in
columns $0,2,3,4,5$ has determinant $1/2$.

A selected vector has seven entries. The strict interior inequalities at indices
$1$ and $3$ would require $\tfrac13<\bar u_3<\tfrac23$, although $\U$ contains no such value.
Thus the necessity direction of the endpoint-excepted strict-interior formulation
fails when full-multiplicity knots are admitted at data sites. This is not a
counterexample in the continuous class. Deleting one occurrence of $\tfrac13$
gives $\V=(0,0,\tfrac13,\tfrac23,\tfrac23,1,1)$ with square determinant $\tfrac12$ under the stated
evaluation convention, despite a failed strict inequality.

\subsection{Membership without multiplicity is insufficient}
\label{revisit:membership}

Even after correcting the endpoint convention, membership alone does not encode
multiplicity. Take
\[
 p=2,\quad m=s=5,\quad
 \U=(0,0,0,\tfrac1{10},\tfrac15,\tfrac12,1,1,1),\qquad
 \mathbf T = (t_0,\dots,t_5)=(0,\tfrac3{10},\tfrac25,\tfrac35,\tfrac7{10},1).
\]
The column of $N_{1,2}^{\U}$ is zero: its support ends at $u_4=1/5$, while all positive
data parameters are at least $3/10$ and its value at $0$ is zero. Thus the square matrix
$\A$ is singular. However,
\[
 \V=(0,0,0,\tfrac1{10},\tfrac12,\tfrac12,1,1,1)
\]
uses only values appearing in $\U$ and satisfies all four interior inequalities:
\[
 0<\tfrac3{10}<\tfrac12,\quad
 0<\tfrac25<\tfrac12,\quad
 \tfrac1{10}<\tfrac35<1,\quad
 \tfrac12<\tfrac7{10}<1.
\]
It uses $1/2$ twice instead of once, so it is not a subsequence of $\U$. This disproves
the sufficiency direction under literal set membership.

\section{The corrected theorem}
\label{sec:corrected}

The analysis above shows that the published statement is correct only once its
conventions are made explicit. 
We now give a precise continuous, clamped formulation of the intended selected-vector
criterion, with occurrence selection and endpoint evaluation specified explicitly.

\begin{theorem}[Corrected Theorem 1]\label{thm:main}
Let $1\le p\le s\le m$ and $0=t_0<t_1<\cdots<t_s=1$, and let $\U$ be a clamped knot vector
with
\[
 u_0=\cdots=u_p=0<u_{p+1},\qquad
 u_m<u_{m+1}=\cdots=u_{m+p+1}=1,
\]
and with every interior knot of multiplicity at most $p$. Let $\A$ and $\Sp(\U)$ be the
collocation matrix and spline space of Section~\ref{sec:notation}. Then $\A$ has full row
rank $s+1$ if and only if one can delete $m-s$ interior knot occurrences from $\U$ to obtain
\[
 \V=(\bar u_0\le\cdots\le\bar u_{s+p+1}),\qquad
 \bar u_0=\cdots=\bar u_p=0,\qquad
 \bar u_{s+1}=\cdots=\bar u_{s+p+1}=1,
\]
such that
\begin{equation}\label{eq:sw}
 \bar u_i<t_i<\bar u_{i+p+1},\qquad 1\le i\le s-1.
\end{equation}
Equivalently, $N_{i,p}^{\V}(t_i)>0$ for every $0\le i\le s$.
\end{theorem}

\noindent Each hypothesis in Theorem~\ref{thm:main} is essential; the examples in
Section~\ref{sec:revisited} show that none can be dropped.
\begin{itemize}
\item \emph{Endpoints.} Since $\bar u_0=t_0=0$ and $\bar u_{s+p+1}=t_s=1$, strict
inequalities at the two endpoints are impossible; the clamps must be handled by the
endpoint convention \eqref{eq:endpoint}, as in \eqref{eq:sw}.
\item \emph{Continuity.} If interior multiplicities equal to $p+1$ are admitted, a
B-spline can be nonzero at an interior endpoint of its support, and the strict
formulation can fail even when $\A$ has full row rank, as in the two-breaks example
\eqref{revisit:two-breaks}. The continuous setting with multiplicities at most $p$ is
therefore essential.
\item \emph{Multiplicity of the selection.} The reduced vector must be an occurrence
subsequence, not merely a vector with values in $\U$; the example in
Section~\ref{revisit:membership} shows that membership alone is too weak.
\end{itemize}
The normalization to $[0,1]$ is only a convention: an affine change of parameter gives
the corresponding statement on any nondegenerate interval $[a,b]$.

When Wang et al.~\cite{wang2008reducing} apply the criterion in their lofting construction,
the knot vector is clamped and the interior knot multiplicities are bounded by the degree
(there they equal $1$), so the hypotheses of Theorem~\ref{thm:main} are satisfied. The
criterion therefore holds there, and the lofted curves are unaffected by the defect in the
published proof.

\section{Proof}
\label{sec:proof}

We use three standard B-spline facts: support and positivity, knot insertion,
and the square Schoenberg--Whitney theorem
\cite{deboor2001}. In the setting of Theorem~\ref{thm:main},
\begin{equation}\label{eq:support}
 N_{i,p}^{\U}(t)>0\ \Longleftrightarrow\ u_i<t<u_{i+p+1}
 \quad(0<t<1),\quad
 N_{i,p}^{\U}(0)=\delta_{i0},\quad N_{i,p}^{\U}(1)=\delta_{im}.
\end{equation}

\begin{figure}[htbp]
\centering
\includegraphics[width=0.45\linewidth]{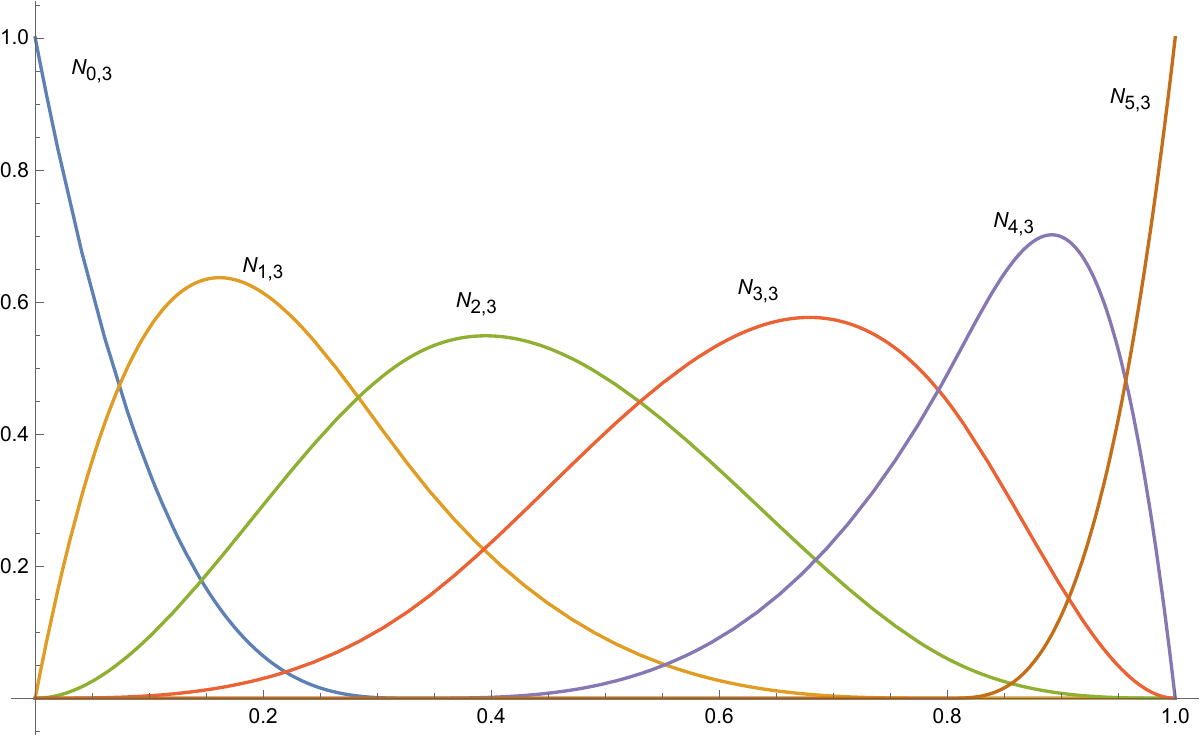}
\caption{The degree-3 B-splines $N_{0,3}^{\U}$ and $N_{5,3}^{\U}$ for the clamped knot vector
$\U=(0,0,0,0,\tfrac13,\tfrac45,1,1,1,1)$, illustrating \eqref{eq:support}: $N_{0,3}^{\U}$
is supported on $[0,\tfrac13]$ and equals $1$ at $t=0$, while
$N_{5,3}^{\U}$ is positive exactly on $(\tfrac45,1)$ and takes the value $1$ at $t=1$.}
\label{fig:basis}
\end{figure}

Figure~\ref{fig:basis} illustrates \eqref{eq:support} on a cubic knot vector.
By continuity, a B-spline vanishes at an interior endpoint of its support. The square
Schoenberg--Whitney theorem then states that the square collocation matrix on $\V$ is
nonsingular precisely when its diagonal entries are positive, with the endpoint convention
in \eqref{eq:support}.

\begin{proof}[Proof of Theorem~\ref{thm:main}]
We prove sufficiency and necessity separately.

\noindent \textbf{Sufficiency.}
Suppose such a $\V$ exists, and set
$\B=[N_{i,p}^{\V}(t_j)]_{j,i=0}^{s}$.
The endpoint values and \eqref{eq:sw} give a positive diagonal, so the square
Schoenberg--Whitney theorem makes $\B$ nonsingular. By knot insertion there
is a matrix $\mathbf C\in\mathbb R^{(m+1)\times(s+1)}$ satisfying
\[
 N_{i,p}^{\V}(t)=\sum_{\ell=0}^{m}C_{\ell i}N_{\ell,p}^{\U}(t),
 \qquad 0\le t\le1.
\]
Evaluation at the data parameters gives $\B=\A\mathbf C$. Hence
\[
 s+1=\rank\B=\rank(\A\mathbf C)\le\rank\A\le s+1,
\]
i.e., $\rank\A=s+1$.

\medskip
\noindent\textbf{Necessity.}
The proof is in two parts: (1)~select columns, (2)~construct the knots.

\medskip
\noindent\emph{(1) Select columns.}
Suppose $\rank\A=s+1$. Choose increasing column indices
$i_0<\cdots<i_s$ such that
\[
 \mathbf M=\bigl[N_{i_k,p}^{\U}(t_j)\bigr]_{j,k=0}^{s}
\]
is nonsingular. The first and last rows of $\A$ are the coordinate rows
$e_0^{\mathsf T}$ and $e_m^{\mathsf T}$. Thus $i_0=0$ and $i_s=m$.
We claim
\begin{equation}\label{eq:minor}
 u_{i_k}<t_k<u_{i_k+p+1},\qquad 1\le k\le s-1.
\end{equation}

We prove \eqref{eq:minor} by contradiction. If $t_k\le u_{i_k}$, then for $j\le k$ and $\ell\ge k$,
\[
 t_j\le t_k\le u_{i_k}\le u_{i_\ell}.
\]
All entries $M_{j\ell}$ in this block vanish by \eqref{eq:support}.
The possible endpoint $t_j=0$ causes no exception because $i_\ell\ge i_k>0$.
Consequently the first $k+1$ rows are supported in just the first $k$ columns;
they are dependent, contradicting nonsingularity.

If $t_k\ge u_{i_k+p+1}$, then for $j\ge k$ and $\ell\le k$,
\[
 t_j\ge t_k\ge u_{i_k+p+1}\ge u_{i_\ell+p+1}.
\]
Again all these entries vanish. At $t_j=1$ there is no exception because
$i_\ell\le i_k<m$. The last $s-k+1$ rows are supported in just the last
$s-k$ columns, another contradiction. This proves \eqref{eq:minor}.

\medskip
\noindent\emph{(2) Construct the knots.}
Define
\begin{equation}\label{eq:construction}
\displaystyle
 \bar u_k=
 \begin{cases}
 0,&0\le k\le p,\\
 u_{i_k},&p+1\le k\le s,\\
 1,&s+1\le k\le s+p+1.
 \end{cases}
\end{equation}
In particular, the retained interior knots are
$u_{i_{p+1}},\ldots,u_{i_s}$; this list is empty when $s=p$.
Since $i_k\ge k$ and $i_s=m$, their indices are distinct elements of
$\{p+1,\ldots,m\}$. Thus \eqref{eq:construction} is a clamped subsequence of
$\U$, with the required $s-p$ interior knot occurrences. It inherits the
interior multiplicity bound.

For $1\le k\le s-1$, the lower inequality $\bar u_k<t_k$ is automatic if
$k\le p$, and otherwise follows from \eqref{eq:minor}.
For the upper inequality, if $k+p+1\ge s+1$, then
$\bar u_{k+p+1}=1>t_k$. Otherwise, since $i_0<i_1<\cdots<i_s$ is a strictly
increasing sequence of integers, $i_{k+r}\ge i_k+r$ for every integer $r>0$; taking
$r=p+1$ gives $i_{k+p+1}\ge i_k+p+1$, whence
\[
 \bar u_{k+p+1}=u_{i_{k+p+1}}
       \ge u_{i_k+p+1}>t_k.
\]
Therefore \eqref{eq:sw} holds, completing the proof.
\end{proof}

The key step is \eqref{eq:construction}: selecting independent columns alone
does not produce a coarser knot vector, but this formula does. The necessity
argument uses only support and elementary linear algebra.

\section{Conclusions}
\label{sec:conclusion}

In this paper we have examined the rectangular Schoenberg--Whitney criterion of
Wang et al.~\cite{wang2008reducing}, which characterizes full row rank of a rectangular
B-spline collocation matrix and underlies energy-minimizing lofting on a common knot
vector. We have shown that the published proof is not valid as written: on an admissible
input one of its inequalities fails, and the endpoint and knot-multiplicity conventions on
which the criterion depends are left implicit. We have stated the corrected criterion for
clamped B-splines, with all hypotheses made explicit, and proved it independently:
sufficiency is a short space-inclusion argument based on knot insertion, while necessity
works by selecting independent columns and reconstructing the reduced knot vector. Unlike the
published argument, the new proof uses no recursive column-selection construction. The
criterion is therefore valid in the continuous, clamped, endpoint-aware setting, and each
of its hypotheses is essential.

% ----------------------------------

\end{document}